\documentclass[reqno]{amsart}
\usepackage{amsfonts}
\usepackage{amssymb}
\usepackage[T1]{fontenc}
\usepackage{newtxtext,newtxmath}
\usepackage{mathtools}
\usepackage{enumitem}
\usepackage[hidelinks]{hyperref}
\usepackage[nameinlink,noabbrev]{cleveref}
\usepackage{microtype}

\numberwithin{equation}{section}
\setlist{nosep,leftmargin=2.1em}
\newtheorem{theorem}{Theorem}[section]
\newtheorem{lemma}[theorem]{Lemma}
\newtheorem{proposition}[theorem]{Proposition}
\theoremstyle{definition}
\newtheorem{definition}[theorem]{Definition}
\newtheorem{example}[theorem]{Example}
\theoremstyle{remark}
\newtheorem{remark}[theorem]{Remark}
\crefname{theorem}{Theorem}{Theorems}
\crefname{lemma}{Lemma}{Lemmas}
\crefname{proposition}{Proposition}{Propositions}
\crefname{definition}{Definition}{Definitions}
\crefname{example}{Example}{Examples}
\crefname{remark}{Remark}{Remarks}
\newcommand{\func}[1]{\operatorname{#1}}

\begin{document}
\title[Abelian maximal pattern complexity of two-dimensional words]{Abelian
maximal pattern complexity of two-dimensional words}
\author{Qingcheng Zeng}
\address{School of Mathematical Sciences, Beihang University, Beijing
100083, P. R. China}
\email{qczeng@buaa.edu.cn}
\author{Yumei Xue}
\address{School of Mathematical Sciences, Beihang University, Beijing
100083, P. R. China}
\email{yxue@buaa.edu.cn}
\author{Cheng Zeng}
\address{School of Mathematics and Information Science, Shandong Technology
and Business University, Yantai, Shandong Province 264003, P. R. China}
\email{czeng@sdtbu.edu.cn}

\begin{abstract}
In this paper, we study the maximal pattern complexity of two-dimensional
words up to Abelian equivalence. We establish a lower bound for the Abelian
maximal pattern complexity of two-dimensional words that are non-doubly
periodic by projection under the existence of a transverse recurrence
direction or strong recurrence. We further show that the bound is attained
for every alphabet size. As a consequence, we characterize double
periodicity of strongly recurrent binary words by the boundedness of their
Abelian maximal pattern complexity.
\end{abstract}

\thanks{Yumei Xue is the corresponding author. }
\maketitle

\section{Introduction}

The study of combinatorial complexity of infinite words is a central topic
in symbolic dynamics and combinatorics on words. Beyond classical factor
complexity, Abelian complexity and maximal pattern complexity provide two
natural variants, with many results established in the one-dimensional
setting \cite{CH1973,KZ2002a,KZ2002b,KR2006,RSZ2011,KWZ2015}. These notions
have also been studied in higher dimensions \cite%
{KX2004,KRX2006,QRWX2010,Puzynina2019}, but Abelian maximal pattern
complexity has received relatively little attention. In this paper, we study
Abelian maximal pattern complexity for two-dimensional words.

\subsection{Factor and Abelian complexity}

\ \thinspace

Let $\mathbb{A}$ be a finite nonempty set. We write $\mathbb{A}^{\ast },%
\mathbb{A}^{\mathbb{N}},\mathbb{A}^{\mathbb{Z}}$ and $\mathbb{A}^{\mathbb{Z}%
^{2}}$, respectively, for the set of finite words, the set of (right)
infinite words, the set of bi-infinite words and the set of two-dimensional
words over the alphabet $\mathbb{A}$, where $\mathbb{N}=\{0,1,2,\ldots \}.$
For an infinite word $x=x(0)x(1)x(2)\cdots \in \mathbb{A}^{\mathbb{N}}$ with 
$x_{i}\in \mathbb{A}$ for all $i\in \mathbb{N}$, we denote by $\mathcal{F}%
_{x}(n)$ the set of all \emph{factors} of $x$ of length $n$, i.e., the set
of all finite words of the form $x(i)x(i+1)\cdots x(i+n-1)$ with $i\in 
\mathbb{N}$. The \emph{factor complexity }of $x$ is defined by%
\begin{equation*}
p_{x}:\mathbb{N}\longrightarrow \mathbb{N},\text{\ }p_{x}(n)=\#\mathcal{F}%
_{x}(n),
\end{equation*}%
where $\#\mathcal{F}_{x}(n)$ denotes the cardinality of $\mathcal{F}_{x}(n)$.

Recall that two words $u,v\in \mathbb{A}^{\ast }$ are said to be \emph{%
Abelian equivalent}, denoted by $u\sim _{\mathrm{ab}}v$, if and only if $%
\left\vert u\right\vert _{a}=\left\vert v\right\vert _{a}$ for all $a\in 
\mathbb{A}$, where $\left\vert u\right\vert _{a}$ denotes the number of
occurrences of the letter $a$ in $u$. A direct check shows that $\sim _{%
\mathrm{ab}}$ is an equivalence relation on $\mathbb{A}^{\ast }$. Let 
\begin{equation*}
\mathcal{F}_{x}^{\mathrm{ab}}(n):=\mathcal{F}_{x}(n)/{\sim _{\mathrm{ab}}}
\end{equation*}%
be the set of Abelian equivalence classes of factors of length $n$. The 
\emph{Abelian complexity }of $x$ is given by 
\begin{equation*}
p_{x}^{\mathrm{ab}}:\mathbb{N}\longrightarrow \mathbb{N}\text{, }p_{x}^{%
\mathrm{ab}}(n)=\#\mathcal{F}_{x}^{\mathrm{ab}}(n).
\end{equation*}%
The notion of Abelian complexity has been extensively studied; see, e.g., 
\cite{RSZ2011,FP2023}.

There are close parallels between factor complexity and Abelian complexity.
For example, both characterize periodic bi-infinite words \cite%
{CH1973,MH1940}. Recall that an infinite word $x$ is \emph{periodic} if
there exists a positive integer $p$ such that $x(i+p)=x(i)$ for all $i\in 
\mathbb{N}$, and \emph{ultimately periodic} if $x(i+p)=x(i)$ for all
sufficiently large $i$. An infinite word is \emph{aperiodic} if it is not
ultimately periodic. Moreover, by Morse and Hedlund \cite{MH1940} and Coven
and Hedlund \cite{CH1973}, respectively, $x$ is Sturmian if and only if 
\begin{equation*}
p_{x}(n)=n+1\text{ for all }n\geq 0,
\end{equation*}%
\ equivalently, if and only if 
\begin{equation*}
p_{x}^{\mathrm{ab}}(n)=2\text{ for all }n\geq 1.
\end{equation*}%
Factor complexity also gives a characterization of ultimately periodic words 
\cite{MH1940}, while Abelian complexity does not. In fact, every Sturmian
word and the ultimately periodic word $01^{\infty }=0111\cdots $\ have the
same constant Abelian complexity $2$.

The notion of Abelian complexity has also been studied for two-dimensional
words. Let $x\in \mathbb{A}^{\mathbb{Z}^{2}}$. For $m,n\geq 1$, the \emph{%
Abelian complexity} $p_{x}^{\mathrm{ab}}(m,n)$ is defined as the number of
Abelian equivalence classes of the rectangular blocks 
\begin{equation*}
x[t+[0,m-1]\times \lbrack 0,n-1]],\text{\ }t\in \mathbb{Z}^{2}.
\end{equation*}%
Here two finite blocks are Abelian equivalent if they contain the same
number of occurrences of each letter.

Puzynina \cite{Puzynina2019} proved that if $x$ is recurrent and 
\begin{equation*}
p_{x}^{\mathrm{ab}}(m,n)\leq 2\text{\ for\ all\ }m,n\geq 1,
\end{equation*}%
then $x$ has a nonzero period vector. The bound is essentially sharp, since
there exist recurrent aperiodic two-dimensional words satisfying 
\begin{equation*}
p_{x}^{\mathrm{ab}}(m,n)\leq 3\text{\ for all }m,n\geq 1.
\end{equation*}

\subsection{Maximal pattern complexity}

\ \thinspace

A distinct notion of complexity for infinite words, called maximal pattern
complexity, was introduced by Kamae and Zamboni \cite{KZ2002a}.

Let $x\in \mathbb{A}^{\mathbb{N}}$. For $k\geq 1$, let $\Sigma _{k}(\mathbb{N%
}):=\{S\subset \mathbb{N}:\sharp S=k\}.$ We call an element $%
S=\{s_{1}<s_{2}<\cdots <s_{k}\}\in \Sigma _{k}(\mathbb{N})$ a $k$-\emph{%
pattern}. Set%
\begin{equation*}
x[S]:=x(s_{1})x(s_{2})\cdots x(s_{k})\in \mathbb{A}^{k}.
\end{equation*}%
For each $n\in \mathbb{N},$ the word $x[n+S]$ is called an $S$-\emph{factor}
of $x$, where $n+S:=\{n+s_{1},n+s_{2},\ldots ,n+s_{k}\}$. We denote by $%
\mathcal{F}_{x}(S)$ the set of all $S$-factors of $x$. The \emph{pattern
complexity }$p_{x}(S)$ is given by%
\begin{equation*}
p_{x}(S):=\#\mathcal{F}_{x}(S),
\end{equation*}%
and the \emph{maximal pattern complexity }of $x$ is defined by 
\begin{equation*}
p_{x}^{\ast }(k):=\sup_{S\in \Sigma _{k}(\mathbb{N})}p_{x}(S).
\end{equation*}

Kamae and Zamboni \cite{KZ2002a} revealed that maximal pattern complexity
also provides a characterization of ultimately periodic words.

\begin{theorem}[\protect\cite{KZ2002a}]
\label{thm:K-Z}Let $x\in \mathbb{A}^{\mathbb{N}}$. Then the following
statements are equivalent.

\begin{enumerate}
\item[(i)] $x$ is ultimately periodic.

\item[(ii)] $p_{x}^{\ast }(k)$ is uniformly bounded in $k$.

\item[(iii)] $p_{x}^{\ast }(k)<2k$ for some positive integer $k$.
\end{enumerate}
\end{theorem}

Kamae and Rao \cite{KR2006} generalized the connection between items (i) and
(iii) in Theorem \ref{thm:K-Z}. An infinite word $x\in \mathbb{A}^{\mathbb{N}%
}$ is called \emph{periodic by projection} if there exists a nonempty set $%
B\subsetneq \mathbb{A}$ such that 
\begin{equation*}
\pi _{B}(x):=\mathbf{1}_{B}\circ x\in \{0,1\}^{\mathbb{N}}
\end{equation*}%
is ultimately periodic, where $\mathbf{1}_{B}$ denotes the characteristic
function of $B$. A word is \emph{aperiodic by projection} if it is not
periodic by projection. The following result relates maximal pattern
complexity to this notion of periodicity.

\begin{theorem}[\protect\cite{KR2006}]
\label{thm:K-R}Let $x\in \mathbb{A}^{\mathbb{N}}$ be aperiodic by projection
and $\sharp \mathbb{A}=r\geq 2$. Then for every positive integer $k$, $%
p_{x}^{\ast }(k)\geq rk.$
\end{theorem}

Kamae, Rao and Xue \cite{KRX2006} extended maximal pattern complexity to
two-dimensional words. For simplicity, we use the same terminology and
notation for patterns in $\mathbb{N}$ and in $\mathbb{Z}^{2}$; the ambient
index set will always be clear from context.

Let $x\in \mathbb{A}^{\mathbb{Z}^{2}}$ be a two-dimensional word. For $k\geq
1$, let $\Sigma _{k}(\mathbb{Z}^{2}):=\{S\subset \mathbb{Z}^{2}:\#S=k\}.$ An
element $S\in \Sigma _{k}(\mathbb{Z}^{2})$ is called a $k$-\emph{pattern}.
For $S\in \Sigma _{k}(\mathbb{Z}^{2})$ and $t\in \mathbb{Z}^{2}$, set 
\begin{equation*}
x[t+S]:=(x(t+s))_{s\in S}\in \mathbb{A}^{S},
\end{equation*}%
which is called an $S$-\emph{factor} of $x$. Let $\mathcal{F}_{x}(S)\ $be
the set of all $S$-factors of $x$. The \emph{pattern complexity }$p_{x}(S)$
is given by%
\begin{equation*}
p_{x}(S):=\#\mathcal{F}_{x}(S),
\end{equation*}%
and the \emph{maximal pattern complexity }of $x$ is defined by 
\begin{equation*}
p_{x}^{\ast }(k):=\sup_{S\in \Sigma _{k}(\mathbb{Z}^{2})}\sharp \mathcal{F}%
_{x}(S).
\end{equation*}

A nonzero vector $p\in \mathbb{Z}^{2}$ is called a \emph{period} of $x$ if $%
x(t+p)=x(t)$\ for all $t\in \mathbb{Z}^{2}.$ The set of all periods of $x$
is denoted by $\func{Per}(x)$. A word $x\in \mathbb{A}^{\mathbb{Z}^{2}}$ is
called \emph{doubly periodic} if $\func{Per}(x)$ contains two linearly
independent vectors. Equivalently, $\func{rank}\func{Per}(x)=2,$ where $%
\func{rank}H$ denotes the maximal number of linearly independent vectors in
a subgroup $H\leq \mathbb{Z}^{2}$.

In \cite{KRX2006}, a two-dimensional version of the notion of strong
recurrence is introduced. Let $\mathbb{P}^{1}(\mathbb{R})$ denote the real
projective line, i.e., the set of one-dimensional linear subspaces of $%
\mathbb{R}^{2}$. For $q\in \mathbb{R}^{2}\setminus \{0\}$, define 
\begin{equation*}
\func{dir}(q):=\mathbb{R}q\in \mathbb{P}^{1}(\mathbb{R}).
\end{equation*}%
Thus $\func{dir}(q)=\func{dir}(q^{\prime })$ if and only if $q^{\prime
}=\lambda q$ for some $\lambda \in \mathbb{R}\setminus \{0\}$; in
particular, $q$ and $-q$ determine the same direction. Identifying $\mathbb{P%
}^{1}(\mathbb{R})$ with $\mathbb{R}/\pi \mathbb{Z}$, where $[\theta
]:=\theta +\pi \mathbb{Z}$ denotes the equivalence class of $\theta \in 
\mathbb{R}$, we equip it with the metric 
\begin{equation*}
d_{\pi }([\theta ],[\phi ]):=\min_{m\in \mathbb{Z}}|\theta -\phi -m\pi |,%
\text{ }\theta ,\phi \in \mathbb{R}.
\end{equation*}%
A word $x\in \mathbb{A}^{\mathbb{Z}^{2}}$ is called \emph{strongly recurrent}
if there exists $\delta >0$ such that for every $N\geq 1$, there exist two
nonzero vectors $u_{N},v_{N}\in \mathbb{Z}^{2}$ such that 
\begin{equation*}
x[\Sigma _{N}+u_{N}]=x[\Sigma _{N}]=x[\Sigma _{N}+v_{N}]\text{ and }d_{\pi }(%
\func{dir}(u_{N}),\func{dir}(v_{N}))\geq \delta ,
\end{equation*}%
where $\Sigma _{N}=[-N,N]^{2}\cap \mathbb{Z}^{2}$.

Kamae, Rao and Xue also established the following two-dimensional analogue
of Theorem \ref{thm:K-Z}.

\begin{theorem}[\protect\cite{KRX2006}]
Let $x\in \mathbb{A}^{\mathbb{Z}^{2}}$. Then the following statements are
equivalent.

\begin{enumerate}
\item[(i)] $x$ is doubly periodic.

\item[(ii)] $p_{x}^{\ast }(k)$ is uniformly bounded in $k$.

\item[(iii)] $x$ is strongly recurrent and $p_{x}^{\ast }(k)<2k$ for some
positive integer $k$.
\end{enumerate}
\end{theorem}

The strong recurrence condition admits a useful geometric description in
terms of recurrence directions. A direction $\theta \in \mathbb{P}^{1}(%
\mathbb{R})$ is called a \emph{recurrence direction} of $x$ if, for any $%
\varepsilon >0$ and $N\geq 1$, there exists $q=q(\varepsilon ,N)\in \mathbb{Z%
}^{2}\setminus \{0\}$ such that 
\begin{equation*}
x[\Sigma _{N}+q]=x[\Sigma _{N}]\text{ and }d_{\pi }(\func{dir}(q),\theta
)<\varepsilon .
\end{equation*}%
We denote by $\mathcal{R}(x)$ the set of all recurrence directions of $x$.
In \cite{KRX2006}, Kamae, Rao and Xue showed that $x$ is strongly recurrent
if and only if $\sharp \mathcal{R}(x)\geq 2.$

\subsection{Abelian maximal pattern complexity}

\ \thinspace

Kamae, Widmer and Zamboni \cite{KWZ2015} introduced an Abelian analog of
maximal pattern complexity for one-dimensional infinite words.

Let $x\in \mathbb{A}^{\mathbb{N}}$. For a $k$-pattern $S\in \Sigma _{k}(%
\mathbb{N}),$ define 
\begin{equation*}
\mathcal{F}_{x}^{\mathrm{ab}}(S):=\mathcal{F}_{x}(S)/{\sim _{\mathrm{ab}},}
\end{equation*}%
and the associated \emph{Abelian pattern complexity} 
\begin{equation*}
p_{x}^{\mathrm{ab}}(S):=\sharp \mathcal{F}_{x}^{\mathrm{ab}}.
\end{equation*}%
The \emph{Abelian maximal pattern complexity} is defined by%
\begin{equation*}
p_{x}^{\ast \mathrm{ab}}(k):=\sup_{S\in \Sigma _{k}(\mathbb{N})}p_{x}^{%
\mathrm{ab}}(S).
\end{equation*}

Then the following Abelian analogue of Theorem \ref{thm:K-R} holds.

\begin{theorem}[\protect\cite{KWZ2015}]
\label{thm:K-W-Z}Let $\#\mathbb{A}=r\geq 2$ and $x\in \mathbb{A}^{\mathbb{N}%
} $. If $x$ is recurrent and aperiodic by projection, then for each positive
integer $k$,%
\begin{equation*}
p_{x}^{\ast ab}(k)\geq (r-1)k+1.
\end{equation*}%
In the case $r=2$, equality always holds. Moreover, for $k=2$ and general $r$%
, there exists $x$ satisfying the equality.
\end{theorem}

In this paper, we extend the notion of Abelian maximal pattern complexity to
two-dimensional words and study its relation to periodicity.

Let $x\in \mathbb{A}^{\mathbb{Z}^{2}}$. For a $k$-pattern $S\in \Sigma _{k}(%
\mathbb{Z}^{2}),$ define%
\begin{equation*}
\mathcal{F}_{x}^{\mathrm{ab}}(S)=\mathcal{F}_{x}(S)/{\sim _{\mathrm{ab}},}
\end{equation*}%
and the associated \emph{Abelian pattern complexity} 
\begin{equation*}
p_{x}^{\mathrm{ab}}(S)=\sharp \mathcal{F}_{x}^{\mathrm{ab}}.
\end{equation*}%
We define the \emph{Abelian maximal pattern complexity} of $x$ by 
\begin{equation*}
p_{x}^{\ast \mathrm{ab}}(k):=\sup_{S\in \Sigma _{k}(\mathbb{Z}^{2})}p_{x}^{%
\mathrm{ab}}(S).
\end{equation*}

It is clear that, for each positive integer $k$ and each $S\in \Sigma _{k}(%
\mathbb{Z}^{2})$, we have 
\begin{equation*}
p_{x}^{\mathrm{ab}}(S)\leq p_{x}(S)\text{\ and\ }p_{x}^{\ast \mathrm{ab}%
}(k)\leq p_{x}^{\ast }(k).
\end{equation*}

Following the notion of periodicity by projection, we introduce its
two-dimensional counterpart using double periodicity, the notion of
periodicity appearing in the two-dimensional maximal pattern complexity
setting.

For a nonempty subset $B\subsetneq \mathbb{A}$, define 
\begin{equation*}
\pi _{B}(x):=\mathbf{1}_{B}\circ x\in \{0,1\}^{\mathbb{Z}^{2}},\text{\ }%
H_{B}:=\func{Per}(\pi _{B}(x)).
\end{equation*}%
We call $x$ \emph{doubly periodic by projection} if there exists a nonempty
subset $B\subsetneq \mathbb{A}$ such that $\pi _{B}(x)$ is doubly periodic.
Otherwise, $x$ is called \emph{non-doubly periodic by projection}.

If $H_{B}$ has rank one, then all its nonzero elements have one unoriented
direction, denoted by $\phi _{H_{B}}\in \mathbb{P}^{1}(\mathbb{R})$. Let 
\begin{equation*}
D_{\mathrm{bad}}(x):=\{\phi _{H_{B}}:\varnothing \neq B\subsetneq A,\text{\ }%
H_{B}\neq \{0\}\}.
\end{equation*}%
Thus $D_{\mathrm{bad}}(x)$ is finite whenever $x$ is non-doubly periodic by
projection. A recurrence direction in $\mathcal{R}(x)\setminus D_{\mathrm{bad%
}}(x)$ is called \emph{transverse}.

We are now ready to state our main results. The first gives a
two-dimensional analogue of Theorem \ref{thm:K-W-Z} under different
recurrence assumptions. The second concerns the binary case.

\begin{theorem}
\label{thm:main}Let $\#\mathbb{A}=r\geq 2$ and $x\in \mathbb{A}^{\mathbb{Z}%
^{2}}$\ be non-doubly periodic by projection.\ Suppose that either

\begin{enumerate}
\item[(i)] $x$ admits a transverse recurrence direction, or

\item[(ii)] $x$ is strongly recurrent.
\end{enumerate}

Then for each positive integer $k$, we have 
\begin{equation*}
p_{x}^{\ast \mathrm{ab}}(k)\geq (r-1)k+1.
\end{equation*}
\end{theorem}

\begin{remark}
Non-doubly periodicity by projection in Theorem~\ref{thm:main} cannot be
weakened to non-double periodicity, even under strong recurrence. Indeed,
Example~\ref{ex:projection-period} gives a strongly recurrent, non-doubly
periodic two-dimensional word $x$ over a three-letter alphabet which is
doubly periodic by projection. Moreover, for every $k\geq 2$,%
\begin{equation*}
p_{x}^{\ast \mathrm{ab}}(k)=k+2<2k+1=(|\mathbb{A}|-1)k+1.
\end{equation*}
\end{remark}

\begin{theorem}
\label{thm:intro-binary-dichotomy} Let $x\in \{0,1\}^{\mathbb{Z}^{2}}$ be
strongly recurrent. Then $x$ is doubly periodic if and only if $\sup_{k\geq
1}p_{x}^{\ast \mathrm{ab}}(k)<\infty .$
\end{theorem}

For a more precise quantitative version of Theorem \ref%
{thm:intro-binary-dichotomy}, see Theorem \ref{thm:binary-dichotomy} in
Subsection \ref{subsec:main}.

The sharpness of the lower bound in Theorem \ref{thm:K-W-Z} remains open in
general \cite{KWZ2015}. In dimension two, however, the lower bound in
Theorem \ref{thm:main} is attained for every alphabet size. More precisely,
we have the following proposition.

\begin{proposition}
\label{prop:sharp-bound}For every integer $r\geq 2,$ there exists $x\in 
\mathbb{A}^{\mathbb{Z}^{2}}$ with $\mathbb{A}=\{0,1,\ldots ,r-1\}$ such that%
\begin{equation}
H_{B}=\{0\}\text{ for all }\varnothing \neq B\subsetneq \mathbb{A},
\label{eq:H_B-0}
\end{equation}%
$x$ admits a transverse recurrence direction, and%
\begin{equation*}
p_{x}^{\ast \mathrm{ab}}(k)=(r-1)k+1\text{\ for all }k\geq 1.
\end{equation*}%
In particular, $x$ is non-doubly periodic by projection. 
\end{proposition}

The paper is organized as follows. Section \ref{sec:prelim} provides some
notations and gives some preliminary results. In Section \ref{sec:avoid}, we
introduce the finite avoidance condition (FA) and show that each of the two
conditions in Theorem \ref{thm:main} implies (FA). Section \ref{sec:sampling}
develops the combinatorial tools needed for the proof. In Section \ref%
{subsec:lower bound}, we derive the lower bound for the Abelian maximal
pattern complexity from (FA). Section \ref{subsec:main} proves the main
theorems, and establishes the sharpness of the lower bound. Finally, Section %
\ref{sec:ex} gives some examples.

\bigskip

\section{Preliminaries}

\label{sec:prelim}

For each $t\in \mathbb{Z}^{2}$, define the translation map $T^{t}:\mathbb{A}%
^{\mathbb{Z}^{2}}\longrightarrow \mathbb{A}^{\mathbb{Z}^{2}}$ by 
\begin{equation*}
(T^{t}y)(u):=y(u+t),\text{\ }y\in \mathbb{A}^{\mathbb{Z}^{2}},\text{\ }u\in 
\mathbb{Z}^{2}.
\end{equation*}%
Then 
\begin{equation*}
T^{0}=\func{id},\text{\ }T^{s+t}=T^{s}\circ T^{t}\text{\ }(s,t\in \mathbb{Z}%
^{2}).
\end{equation*}%
Hence the family $\{T^{t}\}_{t\in \mathbb{Z}^{2}}$ defines a shift action of 
$\mathbb{Z}^{2}$ on $\mathbb{A}^{\mathbb{Z}^{2}}$.

We equip $\mathbb{A}$ with the discrete topology and $\mathbb{A}^{\mathbb{Z}%
^{2}}$ with the product topology. For $x\in \mathbb{A}^{\mathbb{Z}^{2}},$
let 
\begin{equation*}
X_{x}:=\overline{\{T^{t}x:t\in \mathbb{Z}^{2}\}}\subset \mathbb{A}^{\mathbb{Z%
}^{2}}
\end{equation*}%
be the orbit closure in the product topology. Since $\mathbb{A}$ is finite, $%
X_{x}$ is compact. For every nontrivial $B\subsetneq \mathbb{A}$, we have $%
\func{Per}(x)\subset H_{B}$ and $H_{B}=H_{\mathbb{A}\setminus B}$.

We shall use the following immediate consequence of the definition of the
orbit closure.

\begin{lemma}
\label{lem:realize} Let $y\in X_{x}$ and let $F\subset \mathbb{Z}^{2}$ be
finite. Then there exists $t\in \mathbb{Z}^{2}$ such that 
\begin{equation*}
(T^{t}x)(f)=y(f)\text{ for every }f\in F.
\end{equation*}
\end{lemma}

\begin{proof}
The cylinder 
\begin{equation*}
\mathcal{U}_{F}(y):=\{z\in X_{x}:z|_{F}=y|_{F}\}
\end{equation*}%
is a nonempty open neighbourhood of $y$. Since the orbit $\{T^{t}x:t\in 
\mathbb{Z}^{2}\}$ is dense in $X_{x}$, the cylinder contains an orbit point $%
T^{t}x$.
\end{proof}

For a finite word $u\in \mathbb{A}^{m}$, we denote its \emph{Parikh vector}
or \emph{abelianization} by 
\begin{equation*}
\func{Par}(u):=(|u|_{a})_{a\in \mathbb{A}}\in \mathbb{N}^{\mathbb{A}}.
\end{equation*}%
Then 
\begin{equation*}
u\sim _{\mathrm{ab}}v\text{ if and only if }\func{Par}(u)=\func{Par}(v).
\end{equation*}%
Accordingly, we identify an Abelian equivalence class with its Parikh
vector. For $\mathcal{U}\subset \mathbb{A}^{m}$, set $\func{Par}(\mathcal{U}%
)=\{\func{Par}(u):u\in \mathcal{U}\}$.

If $\mathbb{A}=\{0,1\}$ and $u\in \mathbb{A}^{k}$, then 
\begin{equation*}
\func{Par}(u)=(k-j,j)
\end{equation*}%
for some $j\in \{0,1,\ldots ,k\}$. Hence there are at most $k+1$ Abelian
equivalence classes of words of length $k$. Therefore, for $x\in \mathbb{A}^{%
\mathbb{Z}^{2}}$, we have 
\begin{equation}
p_{x}^{\ast \mathrm{ab}}(k)\leq k+1.  \label{eq:binary-upper}
\end{equation}

We recall the equivalent characterizations of strong recurrence of
two-dimensional words.

\begin{lemma}[\protect\cite{KRX2006}]
\label{lem:strong-two-directions} The word $x\in \mathbb{A}^{\mathbb{Z}^{2}}$
is strongly recurrent if and only if $\mathcal{R}(x)$ contains at least two
distinct recurrence directions.
\end{lemma}

For $\Omega \subset \mathbb{A}^{\mathbb{N}}$ and finite $S=\{s_{1}<\cdots
<s_{p}\}\subset \mathbb{N}$, define 
\begin{equation*}
\Omega \lbrack S]:=\{(\omega (s_{1}),\ldots ,\omega (s_{p})):\omega \in
\Omega \}\subset \mathbb{A}^{p}.
\end{equation*}%
If $J=\{j_{0}<j_{1}<\cdots \}\subset \mathbb{N}$ is infinite, define 
\begin{equation*}
\Omega \lbrack J]:=\{(\omega (j_{n}))_{n\geq 0}:\omega \in \Omega \}\subset 
\mathbb{A}^{\mathbb{N}}.
\end{equation*}

We shall use the following Ramsey-type lemma.

\begin{lemma}[\protect\cite{KWZ2015}]
\label{lem:Kamae-ramsey} Let $\Omega \subset \mathbb{A}^{\mathbb{N}}$ be
nonempty and let $I\subset \mathbb{N}$ be an\ infinite set. Then, for every
positive integer $m$, there exists an infinite subset $J\subset I$ such that 
\begin{equation*}
\Omega \lbrack P]=\Omega \lbrack Q]
\end{equation*}%
whenever $P,Q\subset J$ are finite and $1\leq |P|=|Q|\leq m.$
\end{lemma}

For a finite $E\subset \mathbb{Z}^{2}$, a vector $q\in \mathbb{Z}^{2}$ is
called\ an $E$\emph{-return vector} if 
\begin{equation*}
x(z+q)=x(z)\text{ for every }z\in E.
\end{equation*}%
We\ denote by $R_{E}$ the set of all $E$-return vectors$.$ For $U\subset 
\mathbb{P}^{1}(\mathbb{R}),$ write%
\begin{equation*}
R_{E}(U):=\{q\in R_{E}\setminus \{0\}:\func{dir}(q)\in U\}.
\end{equation*}

\bigskip

\section{Finite projection-period coset avoidance}

In this section, we introduce the finite projection-period coset avoidance
condition (\textrm{FA}) and show that each of the two conditions in Theorem %
\ref{thm:main} implies (\textrm{FA}). This allows the subsequent argument to
be carried out under a single hypothesis (\textrm{FA}).

\label{sec:avoid}

\begin{definition}
\label{def:FA} We say that $x$ satisfies \emph{finite projection-period
coset avoidance} (\textrm{FA}) if, for every finite $E\subset \mathbb{Z}^{2}$
and every finite family 
\begin{equation*}
\mathcal{C}=\{c_{\ell }+H_{B_{\ell }}:1\leq \ell \leq M\},
\end{equation*}%
where $M\geq 0$, $c_{\ell }\in \mathbb{Z}^{2}$, and $\varnothing \neq
B_{\ell }\subsetneq \mathbb{A}$, there exists an $E$-return vector $q\in 
\mathbb{Z}^{2}$ such that 
\begin{equation*}
q\notin \bigcup_{C\in \mathcal{C}}C.
\end{equation*}%
For $M=0$ the union is empty.
\end{definition}

\begin{lemma}
\label{lem:infinity-return}Suppose that $\theta \in \mathcal{R}(x)$, $%
E\subset \mathbb{Z}^{2}$ is finite, and $U\subset \mathbb{P}^{1}(\mathbb{R})$
is an open neighbourhood of $\theta $. Then $R_{E}(U)$ is infinite.
Equivalently, $R_{E}(U)$ is unbounded.
\end{lemma}

\begin{proof}
Choose $N_{0}$ large enough such that $E\subset \Sigma _{N_{0}}.$ Since $U\ $%
is an open neighbourhood of $\theta $, by the definition of recurrence
direction, we obtain that for every $N\geq N_{0}$, there exists a nonzero
vector $q_{N}\in \mathbb{Z}^{2}$ such that 
\begin{equation*}
x[\Sigma _{N}+q_{N}]=x[\Sigma _{N}]\text{\ and\ }\func{dir}(q_{N})\in U.
\end{equation*}%
Thus for each $N\geq N_{0},$ $q_{N}\in R_{E}(U)$.

Suppose, to the contrary, that $R_{E}(U)$ is finite. Then there exist a
fixed vector $0\neq q\in \mathbb{Z}^{2}$ and a strictly increasing sequence $%
N_{j}\rightarrow \infty $ such that $q_{N_{j}}=q$ for all $j\geq 1.$ For
arbitrary $z\in \mathbb{Z}^{2}$, take $j$ large enough such that $z\in
\Sigma _{N_{j}}$. Since $x[\Sigma _{N_{j}}+q]=x[\Sigma _{N_{j}}]$, one has $%
x(z+q)=x(z)$, and therefore $q\in \func{Per}(x)$. It follows that $mq\in
R_{E}(U)$ for all $m\geq 1$. Since the vectors $mq$ are pairwise distinct,
this contradicts the finiteness of $R_{E}(U)$. Hence $R_{E}(U)$ is
infinitely. Since every bounded subset of $\mathbb{Z}^{2}$ is finite, $%
R_{E}(U)$ is unbounded.
\end{proof}

\subsection{A transverse recurrence direction implies (FA)}

\begin{lemma}
\label{lem:transverse-FA} If $x$ is non-doubly periodic by projection and
has a transverse recurrence direction, then $x$ satisfies $(\mathrm{FA})$.
\end{lemma}

\begin{proof}
Let $\theta \in \mathcal{R}(x)\setminus D_{\mathrm{bad}}(x)$. Since $D_{%
\mathrm{bad}}(x)$ is finite, there exists an open neighbourhood $U$ of $%
\theta $ such that 
\begin{equation}
\overline{U}\cap D_{\mathrm{bad}}(x)=\varnothing .  \label{eq:U-Dbad}
\end{equation}%
Fix a finite set $E\subset \mathbb{Z}^{2}$ and a finite family 
\begin{equation*}
\mathcal{C}=\{c_{\ell }+H_{B_{\ell }}:1\leq \ell \leq M\}.
\end{equation*}%
By Lemma \ref{lem:infinity-return}, $R_{E}(U)$ is infinitely.

Fix a coset $c+H_{B}\in \mathcal{C}$. If $H_{B}=\{0\}$, then $c+H_{B}$ is a
single point, and hence $(c+H_{B})\cap R_{E}(U)$ is finite. Suppose that $%
\func{rank}H_{B}=1$. Then $c+H_{B}$ is contained in an affine line $c+%
\mathbb{R}v,$ where $\func{dir}(v)=\phi _{H_{B}}$. If $(c+H_{B})\cap
R_{E}(U) $ is infinite, we could choose distinct $q_{n}$ in this
intersection with $\Vert q_{n}\Vert \rightarrow \infty ,$ where $\left\Vert
\cdot \right\Vert $ denotes the Euclidean norm on $\mathbb{R}^{2}$. Since $%
q_{n}\in c+\mathbb{R}v $, one has%
\begin{equation*}
\func{dir}(q_{n})\rightarrow \func{dir}(v)=\phi _{H_{B}}\in D_{\mathrm{bad}%
}(x).
\end{equation*}%
On the other hand, $\func{dir}(q_{n})\in U$ implies that $\phi _{H_{B}}\in 
\overline{U}$, which contradicts (\ref{eq:U-Dbad}). Thus every coset in $%
\mathcal{C}$ meets $R_{E}(U)$ in only finitely many points. Since $\mathcal{C%
}$ is finite and $R_{E}(U)$ is infinite, there exists $q\in R_{E}\backslash
\bigcup_{C\in \mathcal{C}}C.$ Hence $x$ satisfies (FA).
\end{proof}

\subsection{Strong recurrence implies (FA)}

\begin{lemma}
\label{lem:return-family} Suppose that $\theta \in \mathcal{R}(x)$, $%
E\subset \mathbb{Z}^{2}$ is finite, and $U\subset \mathbb{P}^{1}(\mathbb{R})$
is an open neighbourhood of $\theta $. Then, for every $L\geq 1$, there
exist pairwise distinct $E$-return vectors $u_{1},\ldots ,u_{L}\in \mathbb{Z}%
^{2}$ such that 
\begin{equation*}
\func{dir}(u_{i}-u_{j})\in U\text{ whenever }i\neq j.
\end{equation*}
\end{lemma}

\begin{proof}
Let 
\begin{equation*}
\varpi :S^{1}\longrightarrow \mathbb{P}^{1}(\mathbb{R}),\text{\ }\varpi (\xi
)=\mathbb{R}\xi ,
\end{equation*}%
be the natural projection. Choose open neighbourhoods $U_{1}$ and $U_{0}$ of 
$\theta $ in $\mathbb{P}^{1}(\mathbb{R})$ such that 
\begin{equation}
\overline{U_{1}}\subset U_{0},\text{\ }\overline{U_{0}}\subset U.
\label{eq:nested-direction-arcs}
\end{equation}%
Then for each $s=0,1,$ $\varpi ^{-1}(U_{s})$ is the union of two opposite
open arcs of $S^{1}$.

By Lemma \ref{lem:infinity-return}, $R_{E}(U_{1})$ is unbounded. Hence one
component $I_{1}$ of $\varpi ^{-1}(U_{1})$ determines an open cone 
\begin{equation*}
C_{1}:=\{l\xi :l\in \mathbb{R},\ l>0,\text{ }\xi \in I_{1}\}
\end{equation*}%
such that $R_{E}\cap C_{1}$ is unbounded. Let $I_{0}$ be the component of $%
\varpi ^{-1}(U_{0})$ containing $\overline{I_{1}}$, and set 
\begin{equation*}
C_{0}:=\{l\xi :l\in \mathbb{R},\ l>0,\text{ }\xi \in I_{0}\}.
\end{equation*}%
In particular, by (\ref{eq:nested-direction-arcs}) we have 
\begin{equation*}
\func{dir}(C_{0})\subset U_{0}\subset U.
\end{equation*}%
Since $\overline{I_{1}}$ is a compact subset of the open set $C_{0}$, we have%
\begin{equation*}
d:=\func{dist}(\overline{I_{1}},\mathbb{R}^{2}\setminus C_{0})>0.
\end{equation*}

Let $0<\varepsilon <\min \{1,d\}$, we claim that 
\begin{equation}
v\in C_{1},\text{\ }\Vert w\Vert \leq \varepsilon \Vert v\Vert \text{\ }%
\Longrightarrow \text{\ }v-w\in C_{0}.  \label{eq:cone-perturb}
\end{equation}%
Indeed, for 
\begin{equation*}
\xi :=\frac{v}{\Vert v\Vert }\in \overline{I_{1}}\text{ and }z:=\frac{w}{%
\Vert v\Vert },
\end{equation*}%
we have $\Vert z\Vert \leq \varepsilon <d$, and hence $\xi -z\in C_{0}$. As $%
C_{0}$ is closed under multiplication by positive scalars, we obtain 
\begin{equation*}
v-w=\Vert v\Vert (\xi -z)\in C_{0}.
\end{equation*}

Now choose $u_{1},\ldots ,u_{L}\in R_{E}\cap C_{1}$ inductively such that%
\begin{equation}
\Vert u_{j}\Vert >\varepsilon ^{-1}\max_{i<j}\Vert u_{i}\Vert ,\text{ }2\leq
j\leq L.  \label{eq:return-scale-separation}
\end{equation}%
which is possible because $R_{E}\cap C_{1}$ is unbounded. If $i<j$, then $%
\Vert u_{i}\Vert <\varepsilon \Vert u_{j}\Vert $, thus (\ref{eq:cone-perturb}%
) implies $u_{j}-u_{i}\in C_{0}$. Therefore 
\begin{equation*}
\func{dir}(u_{j}-u_{i})\in U.
\end{equation*}%
Since directions are unoriented, we also have $\func{dir}(u_{i}-u_{j})\in U$%
. Finally, (\ref{eq:return-scale-separation}) implies that $\left\Vert
u_{1}\right\Vert <\cdots <\left\Vert u_{L}\right\Vert $, hence the vectors $%
u_{1},\ldots ,u_{L}$ are pairwise distinct.
\end{proof}

\begin{lemma}
\label{lem:grid-lines} Let $\mathcal{U},\mathcal{V}\subset \mathbb{P}^{1}(%
\mathbb{R})$ with $\mathcal{U}\cap \mathcal{V}=\varnothing $. Suppose that%
\begin{equation*}
\{u_{1},\ldots ,u_{L}\},\text{\ }\{v_{1},\ldots ,v_{L}\}\subset \mathbb{R}%
^{2}\text{ }(L\geq 1)
\end{equation*}%
are pairwise distinct within each family, and satisfy 
\begin{equation*}
\func{dir}(u_{i}-u_{i^{\prime }})\in \mathcal{U}\quad \text{whenever }i\neq
i^{\prime },\qquad \func{dir}(v_{j}-v_{j^{\prime }})\in \mathcal{V}\quad 
\text{whenever }j\neq j^{\prime }.
\end{equation*}%
Then the $L^{2}$ sums $u_{i}+v_{j}$ are pairwise distinct, and every affine
line in $\mathbb{R}^{2}$ contains at most $L$ of them.
\end{lemma}

\begin{proof}
We first show that the sums $u_{i}+v_{j}$ are pairwise distinct. Suppose
that 
\begin{equation*}
u_{i}+v_{j}=u_{i^{\prime }}+v_{j^{\prime }}.
\end{equation*}%
Then $u_{i}-u_{i^{\prime }}=v_{j^{\prime }}-v_{j}$. If $i\neq i^{\prime }$,
the common vector $u_{i}-u_{i^{\prime }}$ is nonzero, hence $j\neq j^{\prime
}$ and its direction belongs to both $\mathcal{U}$ and $\mathcal{V}$. This
contradicts $\mathcal{U}\cap \mathcal{V}=\varnothing $. Thus $i=i^{\prime }$%
, and then $j=j^{\prime }$. Hence all $L^{2}$ sums are distinct.

Let $\mathcal{L}$ be an affine line and suppose that $\mathcal{L}$ contains
more than $L$ of these sums. Set 
\begin{equation*}
S_{\mathcal{L}}:=\{(i,j)\in \{1,\ldots ,L\}^{2}:u_{i}+v_{j}\in \mathcal{L}\}.
\end{equation*}%
Then $\sharp S_{\mathcal{L}}>L.$ As there are only $L$ possible first
indices, the pigeonhole principle yields two distinct elements $%
(i_{1},j),(i_{1},j^{\prime })\in S_{\mathcal{L}},$ and hence 
\begin{equation*}
\func{dir}(\mathcal{L})=\func{dir}(v_{j}-v_{j^{\prime }})\in \mathcal{V}.
\end{equation*}%
Similarly, since there are only $L$ possible second indices, there exist
distinct $(i,j_{1}),(i^{\prime },j_{1})\in S_{\mathcal{L}}.$ Thus 
\begin{equation*}
\func{dir}(\mathcal{L})=\func{dir}(u_{i}-u_{i^{\prime }})\in \mathcal{U}.
\end{equation*}%
This contradicts $\mathcal{U}\cap \mathcal{V}=\varnothing $.
\end{proof}

\begin{proposition}
\label{prop:strong-FA} If $x$ is strongly recurrent and non-doubly periodic
by projection, then $x$ satisfies $(\mathrm{FA})$.
\end{proposition}

\begin{proof}
By Lemma \ref{lem:strong-two-directions}, choose distinct recurrence
directions $\theta ,\eta $ and disjoint open neighbourhoods $U_{\theta
},U_{\eta }\subset \mathbb{P}^{1}(\mathbb{R})$ of $\theta $ and $\eta $,
respectively.

Fix a finite set $E\subset \mathbb{Z}^{2}$ and a finite family 
\begin{equation*}
\mathcal{C}=\{c_{\ell }+H_{B_{\ell }}:1\leq \ell \leq M\}.
\end{equation*}%
If $M=0$, then $q=0$ is an $E$-return vector and there is nothing to avoid.
Assume $M\geq 1$. Since $x$ is non-doubly periodic by projection, we have%
\begin{equation*}
\func{rank}H_{B_{\ell }}\leq 1\text{ for every }1\leq \ell \leq M.
\end{equation*}%
Hence each coset $c_{\ell }+H_{B_{\ell }}$ is contained in an affine line $%
\mathcal{L}_{\ell }\subset \mathbb{R}^{2}$.

Let $L:=M+1$. By Lemma \ref{lem:return-family}, there exist $L$ pairwise
distinct $E$-return vectors $u_{1},\ldots ,u_{L}$ such that 
\begin{equation*}
\func{dir}(u_{i}-u_{i^{\prime }})\in U_{\theta }\text{ whenever }i\neq
i^{\prime }.
\end{equation*}%
Set 
\begin{equation*}
E^{\prime }:=\bigcup_{i=1}^{L}(E+u_{i}),
\end{equation*}%
which is finite. Applying Lemma \ref{lem:return-family} to $E^{\prime }$ and
the direction $\eta $, we obtain $L$ pairwise distinct $E^{\prime }$-return
vectors $v_{1},\ldots ,v_{L}$ such that 
\begin{equation*}
\func{dir}(v_{j}-v_{j^{\prime }})\in U_{\eta }\quad \text{whenever }j\neq
j^{\prime }.
\end{equation*}%
For $1\leq i,j\leq L$, define $q_{ij}:=u_{i}+v_{j}$. By Lemma \ref%
{lem:grid-lines}, the $L^{2}$ vectors $q_{ij}$ are distinct and each affine
line $\mathcal{L}_{\ell }$ contains at most $L$ of them. Therefore, we have 
\begin{equation*}
\#\left( \{q_{ij}\}_{i,j}\cap \bigcup_{\ell =1}^{M}\mathcal{L}_{\ell
}\right) \leq ML<L^{2}.
\end{equation*}%
Hence there exists some $q:=q_{ij}$ that lies outside every $\mathcal{L}%
_{\ell }$, and therefore $q\notin \bigcup_{C\in \mathcal{C}}C.$

If $z\in E$, then $x(z+u_{i})=x(z)$ and $z+u_{i}\in E^{\prime }$, whence 
\begin{equation*}
x(z+q_{ij})=x(z+u_{i}+v_{j})=x(z+u_{i})=x(z).
\end{equation*}%
Thus every $q_{ij}$ is an $E$-return vector, which implies $q\in
R_{E}\backslash \bigcup_{C\in \mathcal{C}}C.$ Hence $x$ satisfies $(\mathrm{%
FA})$.
\end{proof}

\bigskip

\section{Sampling sequences and Ramsey extraction}

\label{sec:sampling}

In this section, we use (FA) and the Ramsey-type lemma (Lemma \ref%
{lem:Kamae-ramsey}) to construct a sampling sequence with the properties
needed for the Abelian lower bound in Subsection \ref{subsec:lower bound}.

For $n\in \mathbb{N}$ and distinct $g_{0},\ldots ,g_{n}\in \mathbb{Z}^{2}$,
define 
\begin{equation*}
\mathcal{W}(g_{0},\ldots ,g_{n}):=\{(y(g_{0}),\ldots ,y(g_{n})):y\in X_{x}\}.
\end{equation*}

\begin{proposition}
\label{prop:safe-sampling} Assume that $x$ satisfies $(\mathrm{FA})$. Then
there exists a sequence%
\begin{equation*}
\mathbf{g}=(g_{n})_{n\geq 0}\text{ }(g_{0}=0)
\end{equation*}%
of pairwise distinct points in $\mathbb{Z}^{2}$ satisfying the following
property: 
\begin{flalign*}
\text{\rm(H$_1$)}
&&
g_j-g_i\notin H_B
&&
\end{flalign*}for every $i<j$ and every $\varnothing \neq B\subsetneq 
\mathbb{A}$.

Moreover, set 
\begin{equation*}
\Omega _{\mathbf{g}}:=\{(y(g_{n}))_{n\geq 0}:y\in X_{x}\}\subset \mathbb{A}^{%
\mathbb{N}}.
\end{equation*}%
Then $\Omega _{\mathbf{g}}$ satisfies the following property: 
\begin{flalign*}
\text{\rm(H$_2$)}
&&
\omega _{0}\omega _{1}\cdots \omega _{i-1}\omega _{i}^{\infty }\in \Omega _{%
\mathbf{g}}
&&
\end{flalign*}for every $\omega =(\omega _{n})_{n\geq 0}\in \Omega _{\mathbf{%
g}}$ and every $i\geq 0.$
\end{proposition}

\begin{proof}
We construct the $g_{0},g_{1},g_{2},\ldots $ inductively such that for each $%
n\geq 0,$ the following two properties hold: 
\begin{equation*}
\begin{aligned} (\mathrm{D_n})\qquad & g_j-g_i\notin H_B,\\ & \hspace{2em}
\text{for all }0\le i<j\le n \text{ and all }\varnothing\ne B\subsetneq A;
\\[1mm] (\mathrm{P_n})\qquad & u_0\cdots u_{j-1}u_j^{\,n-j+1} \in \mathcal
W(g_0,\ldots,g_n),\\ & \hspace{2em} \text{for every }0\le j\le n \text{ and
every } u=u_0\cdots u_j\in\mathcal W(g_0,\ldots,g_j). \end{aligned}
\end{equation*}

Take $g_{0}=0$. Property $(\mathrm{P}_{0})$ is immediate and $(\mathrm{D}%
_{0})$ is vacuous.

Assume that $g_{0},\ldots ,g_{n}$ have been constructed and satisfy $(%
\mathrm{D}_{n})$ and $(\mathrm{P}_{n})$. For every $0\leq j\leq n$ and $%
u=u_{0}\cdots u_{j}\in \mathcal{W}(g_{0},\ldots ,g_{j}),$ $(\mathrm{P}_{n})$
implies that there exists $y_{u,j}\in X_{x}$ such that 
\begin{equation}
(y_{u,j}(g_{0}),\ldots ,y_{u,j}(g_{n}))=u_{0}\cdots u_{j-1}u_{j}^{n-j+1}.
\label{eq:reli_1}
\end{equation}%
By Lemma \ref{lem:realize}, there exists $t_{u,j}\in \mathbb{Z}^{2}$ such
that 
\begin{equation}
(T^{t_{u,j}}x)(g_{m})=y_{u,j}(g_{m})\text{ for all }0\leq m\leq n.
\label{eq:reli_2}
\end{equation}%
Let 
\begin{equation*}
E_{n}:=\{t_{u,j}+g_{n}:0\leq j\leq n,\text{ }u\in \mathcal{W}(g_{0},\ldots
,g_{j})\},
\end{equation*}%
then $E_{n}$ is finite since $\mathbb{A}$ is finite. Define the finite
family of forbidden cosets%
\begin{equation*}
\mathcal{C}_{n}:=\{(g_{i}-g_{n})+H_{B}:0\leq i\leq n,\ \varnothing \neq
B\subsetneq \mathbb{A}\}.
\end{equation*}%
Since $x$ satisfies $(\mathrm{FA})$, there exists an $E_{n}$-return vector $%
q_{n}\in \mathbb{Z}^{2}$ such that%
\begin{equation}
q_{n}\notin \bigcup_{C\in \mathcal{C}_{n}}C.  \label{eq:forbi-q_n}
\end{equation}

Define $g_{n+1}:=g_{n}+q_{n}.$ We first verify $(\mathrm{D}_{n+1}).$ The
pairs $0\leq i<j\leq n$ are covered by $(\mathrm{D}_{n}).$ For $0\leq i\leq
n, $ if $g_{n+1}-g_{i}\in H_{B}$ for some $\varnothing \neq B\subsetneq 
\mathbb{A}$, then $g_{n}+q_{n}-g_{i}\in H_{B}.$ Hence 
\begin{equation*}
q_{n}\in (g_{i}-g_{n})+H_{B},
\end{equation*}%
which contradicts (\ref{eq:forbi-q_n}). Thus $(\mathrm{D}_{n+1})$ holds. In
particular, $g_{0},\ldots g_{n+1}$ are pairwise distinct since $0\in H_{B}$
for every $\varnothing \neq B\subsetneq \mathbb{A}.$

We next verify $(\mathrm{P}_{n+1})$. The case $j=n+1$ is immediate. For each
pair $0\leq j\leq n$ and $u\in \mathcal{W}(g_{0},\ldots ,g_{j})$, (\ref%
{eq:reli_1})--(\ref{eq:reli_2}) implies that 
\begin{equation}
(T^{t_{u,j}}x)(g_{0})\cdots (T^{t_{u,j}}x)(g_{n})=u_{0}\cdots
u_{j-1}u_{j}^{n-j+1}.  \label{eq:reli-u-t}
\end{equation}%
Since $t_{u,j}+g_{n}\in E_{n}$ and $q_{n}$ is an $E_{n}$-return vector, we
have 
\begin{equation*}
x(t_{u,j}+g_{n}+q_{n})=x(t_{u,j}+g_{n}),
\end{equation*}%
and hence 
\begin{equation*}
(T^{t_{u,j}}x)(g_{n+1})=(T^{t_{u,j}}x)(g_{n})=u_{j}.
\end{equation*}%
Together with (\ref{eq:reli-u-t}), we obtain that 
\begin{equation*}
u_{0}\cdots u_{j-1}u_{j}^{n-j+2}\in \mathcal{W}(g_{0},\ldots ,g_{n+1}),
\end{equation*}%
thus $(\mathrm{P}_{n+1})$ holds. The induction is complete. Let $\mathbf{g}%
=(g_{n})_{n\geq 0}$. Then $\mathbf{g}$ is pairwise distinct, and property $(%
\mathrm{H}_{1})$ follows immediately from $(\mathrm{D}_{n})$ for all $n\geq
0 $.

It remains to prove ($\mathrm{H}_{2}$). Note that the map $\Phi
:X_{x}\rightarrow \mathbb{A}^{\mathbb{N}}$, $\Phi (y)=(y(g_{m}))_{m\geq 0},$
is continuous. Then $\Omega _{\mathbf{g}}=\Phi (X_{x})$ is compact, and
therefore closed in $\mathbb{A}^{\mathbb{N}}$.

Let $\omega =(\omega _{n})_{n\geq 0}\in \Omega _{\mathbf{g}}$ and $i\geq 0$.
By property $(\mathrm{P}_{n})$, for every $n\geq i$, there exists $%
v^{(n)}\in X_{x}$ satisfying 
\begin{equation*}
v^{(n)}(g_{m})=%
\begin{cases}
\omega _{m}, & 0\leq m<i, \\ 
\omega _{i}, & i\leq m\leq n.%
\end{cases}%
\end{equation*}%
Thus 
\begin{equation*}
\omega _{0}\cdots \omega _{i-1}\omega _{i}^{\infty }=\lim_{n\rightarrow
\infty }\Phi (v^{(n)})\in \Omega _{\mathbf{g}},
\end{equation*}%
which proves that $\Omega _{\mathbf{g}}$ satisfies ($\mathrm{H}_{2}$).
\end{proof}

\begin{lemma}
\label{lem:tail-subsequence}Let $\Omega \subset \mathbb{A}^{\mathbb{N}}$ be
nonempty. If $\Omega $ satisfies property $\mathrm{(H}_{2}\mathrm{)}$, then
so does $\Omega \lbrack J]$ for every infinite subset $J=\{j_{0}<j_{1}<%
\cdots \}\subset \mathbb{N}$.
\end{lemma}

\begin{proof}
Let $\nu =(\nu _{n})_{n\geq 0}\in \Omega \lbrack J]$ and fix $i\geq 0$. By
the definition of $\Omega \lbrack J],$ there exists $\omega =(\omega
_{n})_{n\geq 0}\in \Omega $ such that $\nu _{n}=\omega _{j_{n}}$ for all $%
n\in \mathbb{N}$.\ Apply property $\mathrm{(H}_{2}\mathrm{)}$ to $\omega $
at the coordinate $j_{i}$, one has 
\begin{equation*}
\widetilde{\omega }:=\omega _{0}\cdots \omega _{j_{i}-1}\omega
_{j_{i}}^{\infty }\in \Omega .
\end{equation*}%
Hence 
\begin{equation*}
\nu _{0}\cdots \nu _{i-1}\nu _{i}^{\infty }=(\widetilde{\omega }%
_{j_{n}})_{n\geq 0}\in \Omega \lbrack J].
\end{equation*}%
Thus $\Omega \lbrack J]$ satisfies property $\mathrm{(H}_{2}\mathrm{).}$
\end{proof}

\begin{proposition}
\label{prop:standard-family} Assume that $x$ satisfies $(\mathrm{FA})$ and
fix $k\geq 1$. Then there exists a sequence $\mathbf{h}=(h_{n})_{n\geq 0}$
of pairwise distinct points in $\mathbb{Z}^{2}$ satisfying $\mathrm{(H}_{1}%
\mathrm{).}$ Set 
\begin{equation*}
\Omega :=\{(y(h_{n}))_{n\geq 0}:y\in X_{x}\}\subset \mathbb{A}^{\mathbb{N}}.
\end{equation*}%
Then $\Omega $ satisfies $\mathrm{(H}_{2}\mathrm{).}$ Moreover, we have 
\begin{equation}
\Omega \lbrack P]=\Omega \lbrack Q]  \label{eq:H3-Ramsey}
\end{equation}%
whenever $P,Q\subset \mathbb{N}$ are finite and $1\leq |P|=|Q|\leq k+1$.
\end{proposition}

\begin{proof}
By Proposition \ref{prop:safe-sampling}, there exists a sequence $\mathbf{g}%
\ $satisfying $\mathrm{(H}_{1}\mathrm{)}$, whose associated family $\Omega _{%
\mathbf{g}}$ satisfies $\mathrm{(H}_{2}\mathrm{)}$ and is nonempty. Applying
Lemma \ref{lem:Kamae-ramsey} to $\Omega _{\mathbf{g}}$ with $I=\mathbb{N}$
and $m=k+1$, we obtain an infinite subset $J=\{j_{0}<j_{1}<\cdots \}\subset 
\mathbb{N}$ such that 
\begin{equation}
\Omega _{\mathbf{g}}[P^{\prime }]=\Omega _{\mathbf{g}}[Q^{\prime }]
\label{eq:Ramsey-J}
\end{equation}%
whenever $P^{\prime },Q^{\prime }\subset J$ are finite and $1\leq |P^{\prime
}|=|Q^{\prime }|\leq m.$

Set $h_{n}:=g_{j_{n}}$ $(n\in \mathbb{N})$ and $\Omega :=\Omega _{\mathbf{g}%
}[J]$. Then $h_{n}$ are pairwise distinct. Moreover, $\mathbf{h}%
:=(h_{n})_{n\geq 0}$ inherits $\mathrm{(H}_{1}\mathrm{)}$ from $\mathbf{g}$,
while $\Omega $ satisfies $\mathrm{(H}_{2}\mathrm{)}$ by Lemma \ref%
{lem:tail-subsequence}.

Let $P=\{p_{1}<\cdots <p_{s}\}$ and $Q=\{q_{1}<\cdots <q_{s}\}$ be finite
subsets of $\mathbb{N}$ with $1\leq s\leq k+1$. Then by (\ref{eq:Ramsey-J})
we have%
\begin{equation*}
\Omega \lbrack P]=\Omega _{\mathbf{g}}[\{j_{p_{1}},\ldots
,j_{p_{s}}\}]=\Omega _{\mathbf{g}}[\{j_{q_{1}},\ldots ,j_{q_{s}}\}]=\Omega
\lbrack Q].
\end{equation*}%
This proves the proposition.
\end{proof}

\bigskip

\section{Main results}

\subsection{The Abelian lower bound}

\ \thinspace

\label{subsec:lower bound}

We first derive a lower bound for the Abelian maximal pattern complexity
under condition (\textrm{FA}).

Assume that $x$ satisfies $(\mathrm{FA})$ and fix $k\geq 1$ throughout this
subsection. Let $\mathbf{h}=(h_{n})_{n\geq 0}$ and $\Omega $ be as in
Proposition \ref{prop:standard-family}. Thus $\mathbf{h}$ satisfies $\mathrm{%
(H}_{1}\mathrm{)}$, while $\Omega $ satisfies $\mathrm{(H}_{2}\mathrm{)}$
and (\ref{eq:H3-Ramsey})

For each pair $p<q$, define an undirected graph $\Gamma _{p,q}$ with vertex
set $\mathbb{A}$ by joining distinct $a,b\in \mathbb{A}$ whenever there
exists some $\omega \in \Omega $ such that 
\begin{equation*}
\{\omega _{p},\omega _{q}\}=\{a,b\}.
\end{equation*}

The following lemma is a two-dimensional analogue of Lemma 3.6 of \cite%
{KWZ2015}.

\begin{lemma}
\label{lem:letter-graph} For every $p<q$, the graph $\Gamma _{p,q}$ is
connected.
\end{lemma}

\begin{proof}
Suppose that $\Gamma _{p,q}$ is disconnected, and let $B$ be the vertex set
of one connected component. Then $\varnothing \neq B\subsetneq \mathbb{A}$.
By the definition of $\Gamma _{p,q}$, for every $y\in X_{x}$, one has%
\begin{equation}
\mathbf{1}_{B}(y(h_{p}))=\mathbf{1}_{B}(y(h_{q})).  \label{eq:B-membership}
\end{equation}%
Since $T^{t}x\in X_{x}$ for every $t\in \mathbb{Z}^{2}$, (\ref%
{eq:B-membership}) implies 
\begin{equation}
\mathbf{1}_{B}(x(t+h_{p}))=\mathbf{1}_{B}(x(t+h_{q})).  \label{eq:B-menber-t}
\end{equation}%
Let $d:=h_{q}-h_{p}$. As $t+h_{p}$ ranges over $\mathbb{Z}^{2}$, it follows
from (\ref{eq:B-menber-t}) that 
\begin{equation*}
\pi _{B}(x)(z+d)=\pi _{B}(x)(z)\text{\ for all }z\in \mathbb{Z}^{2}.
\end{equation*}%
Hence $d\in H_{B}$, which contradicts $\mathrm{(H}_{1}\mathrm{)}$.
Therefore, $\Gamma _{p,q}$ is connected.
\end{proof}

The next three combinatorial lemmas adapt the strategy of \cite{KWZ2015} to
the present setting.

We associate with $\Omega $ the undirected graph $G=(\mathbb{A},\mathcal{E})$%
, where for distinct $a,b\in \mathbb{A}$,%
\begin{equation*}
\{a,b\}\in \mathcal{E}\text{ if and only if }a^{k}b^{\infty }\in \Omega 
\text{\ or\ }b^{k}a^{\infty }\in \Omega .
\end{equation*}

\begin{lemma}
\label{lem:tail-graph} The graph $G$ is connected.
\end{lemma}

\begin{proof}
It is enough to show that the endpoints of every edge of $\Gamma _{0,1}$ are
connected in $G$. Fix $\{a,c\}\in E(\Gamma _{0,1})$. After interchanging $a$
and $c$ if necessary, we can assume that 
\begin{equation*}
(a,c)\in \Omega \lbrack \{0,1\}].
\end{equation*}%
Let $\sharp \mathbb{A}=r$. By (\ref{eq:H3-Ramsey}) at level $2$, we have 
\begin{equation*}
(a,c)\in \Omega \lbrack \{kr,kr+1\}]\text{ for all }k\in \mathbb{N}.
\end{equation*}%
Take $\eta \in \Omega $ with $\eta _{kr}=a$ and $\eta _{kr+1}=c$, and write $%
\xi =\eta _{0}\eta _{1}\cdots \eta _{kr-1}.$ Applying $\mathrm{(H}_{2}%
\mathrm{)}$ to $\eta $ at the coordinates $kr+1$ and $kr$, respectively,
gives 
\begin{equation*}
\xi ac^{\infty },\xi a^{\infty }\in \Omega .
\end{equation*}

Note that $\xi $ has length $kr$, it follows that there exists a letter $%
b\in \mathbb{A}$ occurs at least $k$ times in $\xi $. We assume that 
\begin{equation*}
0\leq s_{1}<\cdots <s_{k}<kr,\text{\ }\xi (s_{i})=b.
\end{equation*}%
Hence 
\begin{equation*}
b^{k}a\in \Omega \lbrack \{s_{1},\ldots ,s_{k},kr\}].
\end{equation*}%
By (\ref{eq:H3-Ramsey}) at level $k+1$, one has 
\begin{equation*}
b^{k}a\in \Omega \lbrack \{0,1,\ldots ,k\}].
\end{equation*}%
Thus there exists $\zeta \in \Omega $ whose prefix of length $k+1$ is $%
b^{k}a.$ Apply $\mathrm{(H}_{2}\mathrm{)}$ to $\zeta $\ at coordinate $k$,
we obtain that 
\begin{equation}
b^{k}a^{\infty }\in \Omega .  \label{eq:bka}
\end{equation}%
Applying the same argument to $\xi ac^{\infty },$ with $kr+1$ in place of $%
kr,$ we obtain%
\begin{equation}
b^{k}c^{\infty }\in \Omega .  \label{eq:bkc}
\end{equation}

If $b=a$ or $b=c$, then $\{a,c\}\in \mathcal{E}$ by (\ref{eq:bkc}) or (\ref%
{eq:bka}), respectively. Otherwise,\ we have%
\begin{equation*}
\{a,b\},\{b,c\}\in \mathcal{E},
\end{equation*}%
then $a\ $and $c$ are connected in $G$. Thus the endpoints of every edge of $%
\Gamma _{0,1}$ are connected in $G$. Since $\Gamma _{0,1}$ is connected by
Lemma \ref{lem:letter-graph}, $G$ is connected.
\end{proof}

For $a\in \mathbb{A}$, let $e_{a}\in \mathbb{N}^{\mathbb{A}}$ denote the $a$%
-th standard basis vector. For distinct $a,b\in \mathbb{A}$, define 
\begin{equation*}
V_{a,b}:=\{ie_{a}+(k-i)e_{b}:0\leq i\leq k\}.
\end{equation*}

\begin{lemma}
\label{lem:parikh-segment} If $\{a,b\}\in \mathcal{E}$, then 
\begin{equation*}
V_{a,b}\subset \func{Par}\{\Omega \lbrack \{0,1,\ldots ,k-1\}]\}.
\end{equation*}
\end{lemma}

\begin{proof}
By symmetry, assume that $a^{k}b^{\infty }\in \Omega $. For each $0\leq
i\leq k,$ set 
\begin{equation*}
P_{i}:=\{0,1,\ldots ,i-1\}\cup \{k,k+1,\ldots ,2k-i-1\}.
\end{equation*}%
Then $|P_{i}|=k$ and $(a^{k}b^{\infty })[P_{i}]=a^{i}b^{k-i}.$ By (\ref%
{eq:H3-Ramsey}), one has%
\begin{equation*}
a^{i}b^{k-i}\in \Omega \lbrack P_{i}]=\Omega \lbrack \{0,1,\ldots ,k-1\}],
\end{equation*}%
and therefore 
\begin{equation*}
ie_{a}+(k-i)e_{b}\in \func{Par}\{\Omega \lbrack \{0,1,\ldots ,k-1\}]\}.
\end{equation*}%
Since $i$ is arbitrary, the result follows.
\end{proof}

\begin{lemma}
\label{lem:tree-count} Let $T$ be a tree with vertex set $\mathbb{A}$. Then 
\begin{equation*}
\#\bigcup_{\{a,b\}\in E(T)}V_{a,b}=(r-1)k+1.
\end{equation*}
\end{lemma}

\begin{proof}
For an edge $\{a,b\}$, the $k-1$ internal vectors 
\begin{equation*}
ie_{a}+(k-i)e_{b}\text{\ }(1\leq i\leq k-1)
\end{equation*}%
have positive support exactly $\{a,b\}$. Hence internal vectors arising from
distinct edges are distinct. Since a tree on $r$ vertices has $r-1$ edges,
there are $(r-1)(k-1)$ internal vectors in total. Moreover, for every vertex 
$a\in \mathbb{A}$, the pure vector $ke_{a}$ occurs because $a$ is incident
to some edge. The $r$ pure vectors are pairwise distinct and are distinct
from all internal vectors. Thus 
\begin{equation*}
\#\bigcup_{\{a,b\}\in E(T)}V_{a,b}=(r-1)(k-1)+r=(r-1)k+1.
\end{equation*}
\end{proof}

\begin{theorem}
\label{thm:FA-lower} If $x$ satisfies $(\mathrm{FA})$, then 
\begin{equation*}
p_{x}^{\ast \mathrm{ab}}(k)\geq (r-1)k+1\text{\ for each positive integer }k.
\end{equation*}
\end{theorem}

\begin{proof}
Fix $k\geq 1$, and let $(h_{n})_{n\geq 0}$ and $\Omega $ be given by
Proposition \ref{prop:standard-family}. By Lemma \ref{lem:tail-graph}, the
associated graph $G$ of $\Omega $ is connected. Take a spanning tree $T$ of $%
G$. By Lemmas \ref{lem:parikh-segment}--\ref{lem:tree-count}, we have%
\begin{equation}
\#\func{Par}\{\Omega \lbrack \{0,1,\ldots ,k-1\}]\}\geq (r-1)k+1.
\label{eq:Omega-parikh-lower}
\end{equation}

Let%
\begin{equation*}
F:=\{h_{0},h_{1},\ldots ,h_{k-1}\}\subset \mathbb{Z}^{2}.
\end{equation*}%
It follows from the definition of $\Omega $ and Lemma \ref{lem:realize} that 
\begin{equation*}
\Omega \lbrack \{0,1,\ldots ,k-1\}]=\{((T^{t}(x))(h_{0}),\ldots
,(T^{t}(x))(h_{k-1})):t\in \mathbb{Z}^{2}\}.
\end{equation*}%
By taking Parikh vectors, we obtain that%
\begin{equation*}
\func{Par}\{\Omega \lbrack \{0,1,\ldots ,k-1\}]\}=\{\func{Par}%
_{F}(T^{t}x):t\in \mathbb{Z}^{2}\}.
\end{equation*}%
Therefore, we have 
\begin{equation*}
p_{x}^{\mathrm{ab}}(F)=\#\func{Par}\{\Omega \lbrack \{0,1,\ldots ,k-1\}]\}.
\end{equation*}%
Combining this with (\ref{eq:Omega-parikh-lower}), one has 
\begin{equation*}
p_{x}^{\ast \mathrm{ab}}(k)\geq p_{x}^{\mathrm{ab}}(F)\geq (r-1)k+1.
\end{equation*}
\end{proof}

\subsection{Proofs of the main results}

\label{subsec:main}

\begin{proof}[Proof of Theorem \protect\ref{thm:main}]
Under assumption (i), Lemma \ref{lem:transverse-FA} implies that $x$
satisfies $(\mathrm{FA})$, while under assumption (ii), the same conclusion
follows from Proposition \ref{prop:strong-FA}. Thus, in either case, $x$
satisfies $(\mathrm{FA})$, and the conclusion follows from Theorem \ref%
{thm:FA-lower}.
\end{proof}

\begin{proposition}
\label{prop:periodic-upper}If $x$ is doubly periodic, then 
\begin{equation*}
p_{x}^{\ast \mathrm{ab}}(k)\leq \lbrack \mathbb{Z}^{2}:\func{Per}(x)]\text{\
for all }k\geq 1.
\end{equation*}
\end{proposition}

\begin{proof}
Let $F$ be a finite subset of $\mathbb{Z}^{2}$ and write $M:=$ $[\mathbb{Z}%
^{2}:\func{Per}(x)]$. If $t-s\in \func{Per}(x)$, then $T^{t}x=T^{s}x$, and
consequently 
\begin{equation*}
\func{Par}_{F}(T^{t}x)=\func{Par}_{F}(T^{s}x).
\end{equation*}%
Thus the map $t\mapsto \func{Par}_{F}(T^{t}x)$ is constant on cosets of $%
\func{Per}(x)$, and hence has at most $M$ values. Therefore, we have $p_{x}^{%
\mathrm{ab}}(F)\leq M$. Taking the supremum over all $F$ with $|F|=k,$ one
has $p_{x}^{\ast \mathrm{ab}}(k)\leq M$\ for all $k\geq 1.$
\end{proof}

\begin{theorem}
\label{thm:binary-dichotomy} Let $x\in \{0,1\}^{\mathbb{Z}^{2}}$ be strongly
recurrent. Then exactly one of the following alternatives holds.

\begin{enumerate}
\item[(i)] $x$ is doubly periodic, in which case 
\begin{equation*}
p_{x}^{\ast \mathrm{ab}}(k)\leq \lbrack \mathbb{Z}^{2}:\func{Per}(x)]\text{
for all positive integer }k;
\end{equation*}

\item[(ii)] $x$ is not doubly periodic, in which case 
\begin{equation*}
p_{x}^{\ast \mathrm{ab}}(k)=k+1\text{ for all positive integer }k.
\end{equation*}
\end{enumerate}
\end{theorem}

\begin{proof}
If $x$ is doubly periodic, then (i) follows immediately from Proposition \ref%
{prop:periodic-upper}.

Suppose now that $x$ is not doubly periodic. Since $\mathbb{A}=\{0,1\},$ the
two nonempty proper subsets of $\mathbb{A}$ are complementary, and hence 
\begin{equation*}
H_{\{0\}}=H_{\{1\}}=\func{Per}(x).
\end{equation*}%
Since $x$ is not doubly periodic, $\func{rank}\func{Per}(x)\leq 1$. Thus $x$
is non-doubly periodic by projection. By (ii) of Theorem \ref{thm:main}, one
has 
\begin{equation*}
p_{x}^{\ast \mathrm{ab}}(k)\geq k+1.
\end{equation*}%
Together with the binary upper bound (\ref{eq:binary-upper}), we have 
\begin{equation*}
p_{x}^{\ast \mathrm{ab}}(k)=k+1\text{\ for all }k\geq 1,
\end{equation*}%
which proves (ii). Combining (i) and (ii), $p_{x}^{\ast \mathrm{ab}}(k)$ is
uniformly bounded in $k$ precisely when $x$ is doubly periodic.
\end{proof}

\begin{proof}[Proof of Proposition \protect\ref{prop:sharp-bound}]
Choose $\alpha ,\beta >0$ such that $1,\alpha ,\beta $ are linearly
independent over $\mathbb{Q}$, and define 
\begin{equation*}
\ell :\mathbb{Z}^{2}\rightarrow \mathbb{R}\text{, }\ell (m,n)=n-\alpha m%
\text{ }(\left( m,n\right) \in \mathbb{Z}^{2}).
\end{equation*}%
Set%
\begin{equation*}
x(z):=\sharp \{j\in \{1,\ldots ,r-1\}:\ell (z)>j\beta \},\text{ }z\in 
\mathbb{Z}^{2}.
\end{equation*}%
For $1\leq j\leq r-1,$ let%
\begin{equation*}
L_{j}:=\{(u,\alpha u+j\beta ):u\in \mathbb{R}\}.
\end{equation*}%
The lines $L_{1},\ldots ,L_{r-1}$ are parallel and divide $\mathbb{R}^{2}$
into $r$ open strips. By the rational independence of $1,\alpha ,\beta $,
one has $\mathbb{Z}^{2}\cap L_{j}=\varnothing $ for all $1\leq j\leq r-1$,
and hence $x(z)$ is precisely the label of the strip containing $z$.

Set 
\begin{equation*}
D:=\ell (\mathbb{Z}^{2})=\mathbb{Z}-\alpha \mathbb{Z}.
\end{equation*}%
Since $\alpha $ is irrational, $D$ is dense in $\mathbb{R}$. Then $D$ meets
each of the $r$ intervals 
\begin{equation*}
(-\infty ,\beta ),\qquad (j\beta ,(j+1)\beta )\ (1\leq j\leq r-2),\qquad
((r-1)\beta ,+\infty ),
\end{equation*}%
and therefore every letter of $\mathbb{A}$ occurs in $x$.

We first prove (\ref{eq:H_B-0}). For $\varnothing \neq B\subsetneq \mathbb{A}%
,$ define 
\begin{equation*}
\chi _{B}(u):=\mathbf{1}_{B}(\sharp \{j\in \{1,\ldots ,r-1\}:u>j\beta \}),%
\text{ }u\in \mathbb{R}.
\end{equation*}%
Then 
\begin{equation}
\pi _{B}(x)(z)=\chi _{B}(\ell (z)),\text{ }z\in \mathbb{Z}^{2}.
\label{eq:chiB}
\end{equation}%
Let $S_{B}$ be the discontinuity set of $\chi _{B}$, then $\varnothing \neq
S_{B}\subsetneq \{j\beta \}_{j=1}^{r-1}$ since $B$ is nonempty and proper.

Suppose $q\in H_{B}$. Then, for every $z\in \mathbb{Z}^{2}$, by (\ref%
{eq:chiB}), 
\begin{equation*}
\chi _{B}(\ell (z)+\ell (q))=\chi _{B}(\ell (z+q))=\pi _{B}(x)(z+q)=\pi
_{B}(x)(z)=\chi _{B}(\ell (z)).
\end{equation*}%
It follows that 
\begin{equation}
\chi _{B}(u+\ell (q))=\chi _{B}(u)\text{, }u\in D.  \label{eq:chiB+Delta}
\end{equation}%
Since the two sides of (\ref{eq:chiB+Delta}) are step functions with
finitely many discontinuities and agree on the dense subset $D,$ their
discontinuity sets coincide. Thus one has $S_{B}-\ell (q)=S_{B},$ which
implies $\ell (q)=0.$ As $\alpha $ is irrational and $q\in \mathbb{Z}^{2},$
we have $q=0$, and hence 
\begin{equation*}
H_{B}=\{0\}.
\end{equation*}

Next we present a recurrence direction. Let 
\begin{equation*}
\theta :=\func{dir}(1,\alpha ).
\end{equation*}%
For $N\geq 1,$ let 
\begin{equation*}
\eta _{N}:=\min_{\substack{ z\in \Sigma _{N} \\ 1\leq j\leq r-1}}\left\vert
\ell (z)-j\beta \right\vert >0,
\end{equation*}%
where positivity follows from the fact that $\mathbb{Z}^{2}\cap
L_{j}=\varnothing .$ Let $v_{n}/u_{n}$ be the continued-fraction convergents
of $\alpha $. Then we have%
\begin{equation*}
\left\vert v_{n}-\alpha u_{n}\right\vert \rightarrow 0\text{ and }\func{dir}%
(u_{n},v_{n})\rightarrow \theta .
\end{equation*}%
Given $N\geq 1$ and $\varepsilon >0$, take $n$ large enough such that 
\begin{equation*}
|v_{n}-\alpha u_{n}|<\eta _{N}\text{ and }d_{\pi }(\func{dir}%
(u_{n},v_{n}),\theta )<\varepsilon .
\end{equation*}%
Set $q_{n}:=(u_{n},v_{n})\in \mathbb{Z}^{2}$. Then for every $z\in \Sigma
_{N}$ and $1\leq j\leq r-1$, we have%
\begin{equation*}
|\ell (z)-j\beta |\geq \eta _{N}>|\ell (q_{n})|,
\end{equation*}%
and hence $\ell (z)-j\beta $ and $\ell (z+q_{n})-j\beta $ have the same
sign. Therefore we obtain that 
\begin{equation*}
x[\Sigma _{N}+q_{n}]=x[\Sigma _{N}].
\end{equation*}%
Thus $\theta $\ is a recurrence direction. By (\ref{eq:H_B-0}), we obtain
that $x$ is non-doubly periodic by projection and $D_{\mathrm{bad}%
}=\varnothing .$ Hence the recurrence direction $\theta $ is transverse. By
(i) of Theorem \ref{thm:main}, we have%
\begin{equation}
p_{x}^{\ast \mathrm{ab}}(k)\geq (r-1)k+1\qquad (k\geq 1).
\label{eq:sharp-lower}
\end{equation}

It remains to prove the reverse inequality. Let $F\subset \mathbb{Z}^{2}$
with $\sharp F=k.$ For each $f\in F$, the value $x(t+f)$ is determined by
the position of $\ell (t+f)=\ell (t)+\ell (f)$ relative to the thresholds $%
\beta ,2\beta ,\ldots ,(r-1)\beta .$ Thus the values 
\begin{equation*}
j\beta -\ell (f),\text{\ }f\in F,\text{\ }1\leq j\leq r-1,
\end{equation*}%
divide $\mathbb{R}$ into at most $k(r-1)+1$ open intervals. Moreover, no
value $\ell (t)$ $(t\in \mathbb{Z}^{2})$ is one of these boundary points,
since otherwise $\ell (t+f)=j\beta $ for some $f\in F$, contrary to $%
L_{j}\cap \mathbb{Z}^{2}=\varnothing $.

If $\ell (t)$ and $\ell (s)$ lie in the same interval, then 
\begin{equation*}
x(t+f)=x(s+f),\text{ }f\in F.
\end{equation*}%
Hence the $F$-pattern takes at most $k(r-1)+1$ values, and therefore we have%
\begin{equation*}
p_{x}^{\mathrm{ab}}(F)\leq (r-1)k+1.
\end{equation*}%
Together with (\ref{eq:sharp-lower}), we have 
\begin{equation*}
p_{x}^{\ast \mathrm{ab}}(k)=(r-1)k+1\text{\ for all }k\geq 1.
\end{equation*}
\end{proof}

\bigskip

\section{Examples}

\label{sec:ex}

In this section, we first give an example showing that the non-doubly
periodic by projection assumption in Theorem \ref{thm:main} cannot be
weakened to mere non-doubly periodicity. The construction relies on a basic
result concerning the maximal pattern complexity of bi-infinite binary
words. We therefore begin by recalling the corresponding definition on $%
\mathbb{Z}$ and providing the one-dimensional result that will be used below.

Let $x\in \mathbb{A}^{\mathbb{Z}}.$ For $k\geq 1$, let $\Sigma _{k}(\mathbb{Z%
}):=\{S\subset \mathbb{Z}:\sharp S=k\}.$ We call an element $%
S=\{s_{1}<s_{2}<\cdots <s_{k}\}\in \Sigma _{k}(\mathbb{Z})$ a $k$-\emph{%
pattern}. Set $x[S]:=x(s_{1})x(s_{2})\cdots x(s_{k})\in \mathbb{A}^{k}.$ For
each $z\in \mathbb{Z},$ the word $x[z+S]$ is called an $S$-\emph{factor} of $%
x$. We denote the set of all $S$-factors of $x$ by $\mathcal{F}_{x}(S)$. The 
\emph{pattern complexity }$p_{x}(S)$ is defined by%
\begin{equation*}
p_{x}(S):=\#\mathcal{F}_{x}(S),
\end{equation*}%
and the \emph{maximal pattern complexity }of $x$ is defined by 
\begin{equation*}
p_{x}^{\ast }(k):=\sup_{S\in \Sigma _{k}(\mathbb{Z})}p_{x}(S).
\end{equation*}

\begin{lemma}
\label{lem:binary-recurrent-aperiodic} Let $x\in \{a,b\}^{\mathbb{Z}}$ be
recurrent and aperiodic. Then for every positive integer $k$, 
\begin{equation*}
p_{x}^{\ast \mathrm{ab}}(k)=k+1.
\end{equation*}
\end{lemma}

\begin{proof}
Let $S\in \Sigma _{k}(\mathbb{Z})$. Clearly, for every positive integer $k,$
we have 
\begin{equation}
p_{x}^{\ast \mathrm{ab}}(k)\leq k+1.  \label{eq:binary-upper-bound}
\end{equation}

It remains to prove the reverse inequality. For $N\geq 1$, set $%
I_{N}:=[-N,N]\cap \mathbb{Z}$ and 
\begin{equation*}
R_{N}:=\{q\in \mathbb{Z}:x(j+q)=x(j)\ \text{for every }j\in I_{N}\}.
\end{equation*}%
Since $x$ is recurrent, it follows that each $R_{N}$ is infinite and $%
R_{N+1}\subset R_{N}$ $(N\geq 1).$

We claim that there exists $\varepsilon \in \{-1,1\}$ such that 
\begin{equation}
\sup \{\varepsilon q:q\in R_{N}\}=+\infty ,\text{\ }\forall N\geq 1.
\label{eq:crd}
\end{equation}%
Indeed, otherwise there exist $N_{+},N_{-}\geq 1$ such that $R_{N_{+}}$ is
bounded above and $R_{N_{-}}$ is bounded below. For $N=\max \{N_{+},N_{-}\}$%
, we have%
\begin{equation*}
R_{N}\subset R_{N_{+}}\cap R_{N_{-}}.
\end{equation*}%
Hence $R_{N}$ would be bounded and hence finite, a contradiction.

Define 
\begin{equation*}
\omega (n):=x(\varepsilon n),\text{ }n\in \mathbb{N}.
\end{equation*}%
We first show that $\omega $ is recurrent. Let $u=\omega (i)\omega
(i+1)\cdots \omega (i+\ell -1)$ be an arbitrary finite factor of $\omega $,
and choose $N\geq i+\ell -1$. By (\ref{eq:crd}), take $q\in R_{N}$ such that 
$\varepsilon q>0$ is arbitrarily large, and set $m=i+\varepsilon q.$ For $%
\varepsilon q$ sufficiently large, $m\geq 0$, and $0\leq j<\ell $, we obtain
that%
\begin{equation*}
\omega (m+j)=x(\varepsilon (m+j))=x(\varepsilon (i+j)+q)=x(\varepsilon
(i+j))=\omega (i+j).
\end{equation*}%
Since $\varepsilon q$ can be chosen arbitrarily large, $u$ occurs infinitely
many times in $\omega $. Thus $\omega $ is recurrent.

We next show that $\omega $ is not eventually periodic. Suppose, to the
contrary, that there exist $M\geq 0$ and $p\geq 1$ such that 
\begin{equation*}
\omega (n+p)=\omega (n),\text{ }\forall n\geq M.
\end{equation*}%
Fix $i\in \mathbb{Z}$. Take $N$ large enough such that $i,$\ $i+\varepsilon
p\in I_{N}.$ By (\ref{eq:crd}), choose $q\in R_{N}$ such that $%
n:=\varepsilon (i+q)\geq M.$ Hence we have, 
\begin{equation*}
x(i+\varepsilon p)=x(i+\varepsilon p+q)=x(\varepsilon (n+p))=\omega
(n+p)=\omega (n)=x(\varepsilon n)=x(i+q)=x(i).
\end{equation*}%
Since $i\in \mathbb{Z}$ is arbitrary,\ it follows that $\varepsilon p\neq 0$
is a period of $x$, which contradicts the aperiodicity of $x$. Hence $\omega 
$ is not eventually periodic.

By Theorem \ref{thm:K-W-Z}, we obtain that $p_{\omega }^{\ast \mathrm{ab}%
}(k)=k+1$ for every positive integer $k.$ Let $S\subset \mathbb{N}$ be
finite. Then for every $n\in \mathbb{N}$, 
\begin{equation*}
\func{Par}(\omega \lbrack n+S])=\func{Par}(x[\varepsilon n+\varepsilon S]),
\end{equation*}%
and hence $p_{\omega }^{\mathrm{ab}}(S)\leq p_{x}^{\mathrm{ab}}(\varepsilon
S).$ Taking the supremum over all $k$-patterns $S\subset \Sigma _{k}(\mathbb{%
N}),$ we have 
\begin{equation*}
k+1=p_{\omega }^{\ast \mathrm{ab}}(k)\leq p_{x}^{\ast \mathrm{ab}}(k).
\end{equation*}%
Together with (\ref{eq:binary-upper-bound}), this implies that 
\begin{equation*}
p_{x}^{\ast \mathrm{ab}}(k)=k+1.
\end{equation*}
\end{proof}

For a binary alphabet, non-double periodicity is equivalent to non-doubly
periodicity by projection. Hence a counterexample requires at least three
letters.

\begin{example}
\label{ex:projection-period} Let $\omega \in \{a,b\}^{\mathbb{Z}}$ be
recurrent and aperiodic, and let $\mathbb{A}=\{a,b,c\}$. Define 
\begin{equation*}
x(m,n)=%
\begin{cases}
\omega (m), & n\ \text{even}, \\ 
c, & n\ \text{odd}.%
\end{cases}%
\end{equation*}%
Then $x$ is strongly recurrent. In fact, for every $N$, take a nonzero
return $p_{N}$ of $\omega \lbrack -N,N]$. Then 
\begin{equation*}
x[\Sigma _{N}+(p_{N},0)]=x[\Sigma _{N}]=x[\Sigma _{N}+(0,2)].
\end{equation*}%
Hence $(p_{N},0)$ and $(0,2)$ are two orthogonal return directions.

We have 
\begin{equation*}
\func{Per}(x)=\{0\}\times 2\mathbb{Z}.
\end{equation*}%
Indeed, suppose that $(p,q)\in \func{Per}(x)$. Then $q$ is even, since
otherwise an even row is mapped to a row consisting entirely of $c.$ This
means that $\omega (m+p)=\omega (m)$ for all $m$, and hence aperiodicity of $%
\omega $ forces $p=0$. Therefore, $x$ is not doubly periodic.

On the other hand, for $B=\{a,b\}$, we have 
\begin{equation*}
\pi _{B}(x)(m,n)=%
\begin{cases}
1, & n\ \text{even}, \\ 
0, & n\ \text{odd}.%
\end{cases}%
\end{equation*}%
Hence $(1,0),(0,2)\in H_{B}$ are linearly independent. Thus $\pi _{B}(x)$ is
doubly periodic, and consequently $x$ is doubly periodic by projection.

We now compute the Abelian maximal pattern complexity of $x$. We claim that 
\begin{equation}
p_{x}^{\ast \mathrm{ab}}(k)=k+2\text{ for each positive integer }k.
\label{eq:projection-example-abelian}
\end{equation}

We first prove the upper bound. For a $k$-pattern $F\in \Sigma _{k}(\mathbb{Z%
}^{2}),$ we split $F$ according to the parity of the second coordinate,
i.e., 
\begin{equation*}
F_{0}:=\{(u,v)\in F:v\equiv 0\pmod2\},\text{\ }F_{1}:=\{(u,v)\in F:v\equiv 1%
\pmod2\}.
\end{equation*}%
Set $k_{0}:=|F_{0}|,$ $k_{1}:=|F_{1}|,$ $k_{0}+k_{1}=k.$

Let $t=(t_{1},t_{2})\in \mathbb{Z}^{2}$. If $t_{2}$ is even, then every
point of $t+F_{0}$ lies on an even row and every point of $t+F_{1}$ lies on
an odd row. Hence $x(t+f)=c$ for every $f\in F_{1}$. Therefore, we have%
\begin{equation*}
\func{Par}_{F}(T^{t}x)\in \{(i,k_{0}-i,k_{1}):0\leq i\leq k_{0},\}
\end{equation*}%
where the coordinates are indexed by $a,b,c$, respectively. If $t_{2}$ is
odd, the roles of $F_{0}$ and $F_{1}$ are interchanged. Hence in this case,
we have%
\begin{equation*}
\func{Par}_{F}(T^{t}x)\in \{(j,k_{1}-j,k_{0}):0\leq j\leq k_{1},\}
\end{equation*}%
Consequently, 
\begin{equation*}
p_{x}^{\mathrm{ab}}(F)\leq (k_{0}+1)+(k_{1}+1)=k+2.
\end{equation*}%
Since $F$ is arbitrary, we obtain that 
\begin{equation}
p_{x}^{\ast \mathrm{ab}}(k)\leq k+2.  \label{eq:projection-example-upper}
\end{equation}

It remains to prove that this bound is attained. Since $\omega $ is a
recurrent aperiodic binary word, Lemma \ref{lem:binary-recurrent-aperiodic}
implies that 
\begin{equation*}
p_{\omega }^{\ast \mathrm{ab}}(k)=k+1.
\end{equation*}%
Hence there exists a $k$-pattern $S\in \Sigma _{k}(\mathbb{Z})$ such that $%
p_{\omega }^{\mathrm{ab}}(S)=k+1.$ After translating $S$ if necessary, we
may assume that $0\in S$. Define%
\begin{equation*}
F_{S}:=\{(u,0):u\in S\}\subset \mathbb{Z}^{2},
\end{equation*}%
then $F_{S}\in \Sigma _{k}(\mathbb{Z}^{2}).$

For $t=(t_{1},2q)\in \mathbb{Z}\times 2\mathbb{Z}$, all points of $t+F_{S}$
lie on an even row, and hence 
\begin{equation*}
x(t_{1}+u,2q)=\omega (t_{1}+u)\text{ }(u\in S).
\end{equation*}%
Therefore, 
\begin{equation*}
\func{Par}_{F_{S}}(T^{(t_{1},2q)}x)=(\left\vert \omega \lbrack
t_{1}+S]\right\vert _{a},\left\vert \omega \lbrack t_{1}+S]\right\vert
_{b},0),
\end{equation*}%
thus 
\begin{equation*}
\sharp \{\func{Par}_{F_{S}}(T^{(t_{1},2q)}x):t_{1}\in \mathbb{Z}\}=p_{\omega
}^{\mathrm{ab}}(S)=k+1.
\end{equation*}

On the other hand, for $t=(t_{1},2q+1)$, all points of $t+F_{S}$ lie on an
odd row. Hence 
\begin{equation*}
x(t_{1}+u,2q+1)=c\text{\ }(u\in S),
\end{equation*}%
and the corresponding Parikh vector is $(0,0,k).$ Thus%
\begin{equation*}
p_{x}^{\mathrm{ab}}(F_{S})=k+2.
\end{equation*}%
It follows that 
\begin{equation}
p_{x}^{\ast \mathrm{ab}}(k)\geq k+2.  \label{eq:projection-example-lower}
\end{equation}

Combining (\ref{eq:projection-example-upper}) and (\ref%
{eq:projection-example-lower}), we obtain (\ref%
{eq:projection-example-abelian}).
\end{example}

We next give a natural class of four-letter two-dimensional words that are
strongly recurrent and non-doubly periodic by projection.

\begin{proposition}
\label{prop:product-example} Let $\omega ,\varpi \in \{0,1\}^{\mathbb{Z}}$
be recurrent and aperiodic. Define 
\begin{equation*}
x(m,n)=(\omega (m),\varpi (n))\in \mathbb{A}^{\mathbb{Z}^{2}},
\end{equation*}%
where $\mathbb{A}=\{0,1\}^{2}.$ Then $x$ is strongly recurrent and
non-doubly periodic by projection. Consequently, 
\begin{equation}
p_{x}^{\ast \mathrm{ab}}(k)\geq 3k+1\text{\ }(k\geq 1).
\label{eq:s-lower bound}
\end{equation}
\end{proposition}

\begin{proof}
For every $N\geq 1$, since $\omega $ and $\varpi $ are recurrent, there
exist $p_{N},q_{N}\in \mathbb{Z}\setminus \{0\}$ such that 
\begin{equation*}
\omega \lbrack I_{N}+p_{N}]=\omega \lbrack I_{N}]\text{ and }\varpi \lbrack
I_{N}+q_{N}]=\varpi \lbrack I_{N}].
\end{equation*}%
Therefore, we have%
\begin{equation*}
x[\Sigma _{N}+(p_{N},0)]=x[\Sigma _{N}]=x[\Sigma _{N}+(0,q_{N})].
\end{equation*}%
The two nonzero vectors $(p_{N},0),$\ $(0,q_{N})$ have directions separated
by $\pi /2$. Hence $x$ is strongly recurrent.

For a nonempty $B\subsetneq \mathbb{A}$, assume that $\pi _{B}(x)$ is doubly
periodic. Then $[\mathbb{Z}^{2}:H_{B}]<\infty $, and hence there exist $%
P,Q\geq 1$ such that $(P,0),(0,Q)\in H_{B}.$

Suppose first that $\mathbf{1}_{B}(0,b)\neq \mathbf{1}_{B}(1,b)$ for some $%
b\in \{0,1\}$. Since $\varpi $ is aperiodic, both letters occur in $\varpi $%
. Hence $\varpi (n_{0})=b$ for some $n_{0}\in \mathbb{Z}$. Then for every $%
m\in \mathbb{Z}$, we have 
\begin{equation*}
\mathbf{1}_{B}(\omega (m+P),b)=\pi _{B}(x)(m+P,n_{0})=\pi _{B}(x)(m,n_{0})=%
\mathbf{1}_{B}(\omega (m),b).
\end{equation*}%
Since $\mathbf{1}_{B}(0,b)\neq \mathbf{1}_{B}(1,b),$ it follows that $\omega
(m+P)=\omega (m)$\ $(m\in \mathbb{Z}),$ which contradicts the aperiodicity
of $\omega $.

It remains to consider the case where $\mathbf{1}_{B}(0,b)=\mathbf{1}%
_{B}(1,b)$ for every $b\in \{0,1\}$. Since $\varnothing \neq B\subsetneq A$,
the function $\mathbf{1}_{B}$ is nonconstant, and hence $\mathbf{1}%
_{B}(0,0)\neq \mathbf{1}_{B}(0,1).$ Since $(0,Q)\in H_{B}$, for every $%
m,n\in \mathbb{Z}$ we obtain 
\begin{equation*}
\mathbf{1}_{B}(0,\varpi (n+Q))=\pi _{B}(x)(m,n+Q)=\pi _{B}(x)(m,n)=\mathbf{1}%
_{B}(0,\varpi (n)).
\end{equation*}%
Therefore $\varpi (n+Q)=\varpi (n)$\ $(n\in \mathbb{Z}),$ which contradicts
the aperiodicity of $\varpi $.

Thus $\pi _{B}(x)$ is not doubly periodic. Since $B\subsetneq A$ was
arbitrary, $x$ is non-doubly periodic by projection. By (ii) of Theorem \ref%
{thm:main}, \ we obtain the stated lower bound (\ref{eq:s-lower bound}).
\end{proof}


\bigskip

\begin{thebibliography}{99}
\bibitem{CH1973} E. M. Coven, G. A. Hedlund, Sequences with minimal block
growth, Math. Syst. Theory \textbf{7} (1973) 138--153.

\bibitem{FP2023} G. Fici, S. Puzynina, Abelian combinatorics on words: A
survey, Comput. Sci. Rev. \textbf{47} (2023) 100532.

\bibitem{KR2006} T. Kamae, H. Rao, Maximal pattern complexity over $\ell $
letters, European J. Combin. \textbf{27} (2006) 125--137.

\bibitem{KRX2006} T. Kamae, H. Rao, Y.-M. Xue, Maximal pattern complexity of
two-dimensional words, Theoret. Comput. Sci. \textbf{359} (2006) 15--27.

\bibitem{KX2004} T. Kamae, Y.-M. Xue, Two dimensional word with $2k$ maximal
pattern complexity, Osaka J. Math. \textbf{41}(2) (2004) 257--265.

\bibitem{KWZ2015} T. Kamae, S. Widmer, L. Q. Zamboni, Abelian maximal
pattern complexity of words, Ergodic Theory Dynam. Systems \textbf{35}
(2015) 142--151.

\bibitem{KZ2002a} T. Kamae, L. Q. Zamboni, Sequence entropy and the maximal
pattern complexity of infinite words, Ergodic Theory Dynam. Systems \textbf{%
22} (2002) 1191--1199.

\bibitem{KZ2002b} T. Kamae, L. Q. Zamboni, Maximal pattern complexity for
discrete systems, Ergodic Theory Dynam. Systems \textbf{22} (2002)
1201--1214.

\bibitem{MH1940} M. Morse, G. A. Hedlund, Symbolic dynamics II: Sturmian
trajectories, Amer. J. Math. \textbf{62}(1) (1940) 1--42.

\bibitem{Puzynina2019} S. Puzynina, Aperiodic two-dimensional words of small
abelian complexity, Electron. J. Combin. \textbf{26}(4) (2019) \#P4.15.

\bibitem{QRWX2010} Y.-H. Qu, H. Rao, Z.-Y. Wen, Y.-M. Xue, Maximal pattern
complexity of higher dimensional words, J. Combin. Theory Ser. A \textbf{117}%
(5) (2010), 489--506.

\bibitem{RSZ2011} G. Richomme, K. Saari, L. Q. Zamboni, Abelian complexity
of minimal subshifts, J. Lond. Math. Soc. 83(1) (2011) 79--95.
\end{thebibliography}
\end{document}